\documentclass[a4paper,10pt,reqno]{amsart}
\usepackage{verbatim}
\usepackage{amssymb}

\usepackage{enumerate}
\usepackage[active]{srcltx}
\numberwithin{equation}{section}

\usepackage{t1enc}
\usepackage[utf8x]{inputenc}

\newtheorem{theorem}{Theorem}[section]
\newtheorem{lemma}[theorem]{Lemma}

\newtheorem{corollary}[theorem]{Corollary}

\theoremstyle{definition}
\newtheorem{definition}[theorem]{Definition}

\theoremstyle{remark}

\newcommand{\mc}[1]{\mathcal{#1}}

\newcommand{\setm}{\setminus}
\newcommand{\empt}{\emptyset}
\newcommand{\subs}{\subset}

\def\<{\left\langle}
\def\>{\right\rangle}
\def\br#1;#2;{\bigl[ {#1} \bigr]^ {#2} }

\DeclareMathOperator{\FF}{F}
\DeclareMathOperator{\PSIC}{\psi_c}
\DeclareMathOperator{\lin}{L}

\DeclareMathOperator{\w}{w}

\author[I. Juh\'asz]{Istv\'an Juh\'asz}
\address      { Alfr\'ed Rényi Institute of Mathematics}
\email{juhasz@renyi.hu}

\author[L. Soukup]{Lajos Soukup}
\thanks
  {
   }
\address
      { Alfr{\'e}d R{\'e}nyi Institute of Mathematics%
}
\email{soukup@renyi.hu}

\author[Z. Szentmikl\'ossy]{Zolt\'an Szentmikl\'ossy}
\address{E\"otv\"os University of Budapest}
\email{szentmiklossyz@gmail.com}


\title[On the free set number]{On the free set number of topological spaces and their $G_\delta$-modifications}
\thanks{The research on and preparation of this paper was
supported by  NKFIH grant no.  K 129211.}
\date{\today}

\begin{document}

\begin{abstract}
For a topological space $X$ we propose to call a {\em subset} $S \subs X$ {\em free} in $X$ if it admits a well-ordering
that turns it into a free sequence in $X$. The well-known cardinal function $\FF(X)$ is then definable as
$\sup\{|S| : S \text{ is free in } X\}$ and will be called the free set number of $X$.

We prove several new inequalities involving $\FF(X)$ and $\FF(X_\delta)$, where $X_\delta$ is the $G_{\delta}$-modification of $X$:

\smallskip

\begin{itemize}
\item $\lin(X) \le 2^{2^{\FF(X)}}$ if $X$ is $T_2$ and $\lin(X)\le 2^{\FF(X)}$ if $X$ is $T_3$;

\smallskip

\item $|X|\le 2^{2^{\FF(X) \cdot \PSIC(X)}} \le 2^{2^{\FF(X) \cdot \chi(X)}}$ for any $T_2$-space $X$;

\smallskip

\item $\FF(X_{\delta})\le 2^{2^{2^{\FF(X)}}}$ if $X$ is $T_2$ and $\FF(X_{\delta})\le 2^{2^{\FF(X)}}$ if $X$ is $T_3$.
\end{itemize}

\end{abstract}

\maketitle

\section{Free sequences, free sets, and the free set number}

As is well-known, a transfinite sequence $\<x_\alpha : \alpha < \eta\>$ of points of a
topological space $X$ is said to be a {\em free sequence} in $X$ if the closure of any
initial segment of it is disjoint from the closure of the corresponding final segment,
i.e. $$\overline{\{x_\alpha : \alpha < \beta\}} \cap \overline{\{x_\alpha : \alpha \ge \beta\}} = \emptyset$$
for each $\beta < \eta$. Also, the cardinal function $\FF(X)$ which is defined as the
supremum of those cardinals $\kappa$ for which there is a free sequence in $X$ of length $\kappa$
has been thoroughly investigated.

Unlike many other cardinal functions, $\FF(X)$ so far had
no name, probably because in general the length of a sequence is not a cardinal number.
So, as the aim of this paper is to present several new results concerning $\FF(X)$,
we decided to pass from sequences to sets and call a {\em subset} $S \subs X$ {\em free} in $X$ if it admits a well-ordering,
or equivalently an indexing by ordinals,
that turns it into a free sequence in $X$.
In other words, free sets in $X$ are just the ranges of free sequences in $X$. We shall use $\mc{F}(X)$ to denote the collection of
all free subsets in $X$. Clearly, then $$\FF(X) = \sup\{|S| : S \in \mc{F}(X)\},$$
hence we propose to call $\FF(X)$ the {\em free set number} of $X$.

All our other terminology and notation is standard, as e.g. it is in \cite{Ju}.

\section{Some new inequalities involving the free set number}

We start by presenting a very simple and basic but very general construction of free sets. It has been used
by many authors in numerous particular cases.

\begin{lemma}\label{lm:freeinB}
Assume that $X$ is a space, $A \subs X$, $\,{\kappa}$ is an infinite cardinal, and
$\mc W\subs \tau(X)$,  moreover
\begin{enumerate}[(a)]
\item $\mc W$ is closed under unions of subfamilies of size $<\,\kappa$,
\item $A\setm W\ne \empt$ for each $W\in \mc W$,
\item for each $S \subs A$ with $S\in \mc{F}(X)$ and $|S| < \kappa$ there is $W\in \mc W$
with $\overline{S}\subs W$.
\end{enumerate}
Then $\mc{F}(X) \cap [A]^\kappa \ne \emptyset$, i.e. there is a subset of $A$ of size $\kappa$ that is free in $X$.
\end{lemma}

\begin{proof}
By transfinite recursion on ${\alpha}<{\kappa}$ we shall define points
$x_{\alpha} \in A$ and open sets $W_{\alpha} \in \mc W$ such that
      $\overline {\{x_{\beta}:{\beta}<{\alpha}\}}\subs W_{\alpha}$
      and $W_{\alpha}\cap \{x_{\beta}:{\beta}\ge {\alpha}\}=\empt$.

In the ${\alpha}$th step we clearly have $\{x_{\beta}:{\beta}<{\alpha}\} \in \mc{F}(X)$, hence we may
use (c) to pick $W_{\alpha}\in \mc W$ with $\overline {\{x_{\beta}:{\beta}<{\alpha}\}}\subs W_\alpha$.
By (a) and (b) we can then pick $x_\alpha \in A \setm \bigcup_{{\beta}\le {\alpha}}W_{\beta}$.

Clearly, the set  $\{x_{\alpha}:{\alpha}<{\kappa}\}$ is as required.
\end{proof}

We shall only use Lemma \ref{lm:freeinB} in cases where $\kappa = \lambda^+$ is a successor cardinal.
The generality of Lemma \ref{lm:freeinB} may turn out to be useful in studying the "hat version"
$$\widehat{\FF}(X) = \min \{\kappa : |S| < \kappa \text{ for all } S \in \mc{F}(X)\}$$ of $\FF(X)$
in cases when $\FF(X)$ is a limit cardinal.

\smallskip

Next we introduce two auxiliary cardinal functions whose definitions involve free sets
and that will play important role in what follows.

\begin{definition}
Given a topological space $X$ let
\begin{enumerate}[(i)]
\item ${\mu}(X)=\sup\{\lin(\overline {S}):S\in \mc F(X)\}$,
\item ${\nu}(X)=\sup\{|\overline{S}|:S\in \mc F(X)\}$.
\end{enumerate}
\end{definition}

Clearly, we have $\mu(X) \le \nu(X)$ for any space $X$ and $\nu(X) \le 2^{2^{\FF(X)}}$ if $X$ is $T_2$.

\begin{theorem}\label{tm:LsupLsmall}
For each topological space $X$ we have
\begin{displaymath}
\lin(X)\le \FF(X)\cdot{\mu}(X).
\end{displaymath}
\end{theorem}

\begin{proof}
Let us write ${\lambda}=\FF(X)\cdot {\mu}(X)$ and
assume, on the contrary, that the open cover $\mc U$ of $X$ has no subcover of size $\le \lambda$.
If we put
\begin{displaymath}
\mc W=\Big\{\bigcup \mc V: \mc V\in {[\mc U]}^{{\le \lambda}}\Big\},
\end{displaymath}
then for every $S\in \mc F(X)\}$ there is $W \in \mc W$ with $\overline S \subs W$ because
$\lin(\overline S)\le {\mu}(X)\le \lambda$. Thus we may
apply Lemma \ref{lm:freeinB} with the choices $A = X$, $\mc W$, and $\kappa = \lambda^+$
to obtain a free set in $X$ of size ${\lambda}^+ > \FF(X)$, a contradiction.
\end{proof}

For every subset $S$ of a $T_2$-space $X$ we have $\lin(\overline S)\le |\overline S| \le 2^{2^{|S|}}$,
consequently, ${\mu}(X) \le 2^{2^{\FF(X)}}$. If, on the other hand, $X$ is even $T_3$ then
$\lin(\overline S)\le \w(\overline S) \le 2^{|{S}|}$ for every $S \subs X$, hence ${\mu}(X) \le 2^{\FF(X)}$. From these
we immediately obtain the following result.

\begin{corollary}\label{cr:Lle22F}
(i) If $X$ is any $T_2$-space then
\begin{displaymath}
\lin(X)\le 2^{2^{\FF(X)}}.
\end{displaymath}

\smallskip

(ii) If $X$ is $T_3$ then
\begin{displaymath}
\lin(X)\le 2^{\FF(X)}.
\end{displaymath}
\end{corollary}

Before we turn to our next result, we recall
that  the closed pseudocharacter $\psi_c(p,X)$ of a point $p$ in a space $X$ is defined
as the smallest number of {\em closed} neighborhoods of $p$ whose intersection is $\{p\}$,
see e.g. \cite[p. 8.]{Ju}, 
Of course, $X$ is $T_2$ exactly then when $\psi_c(p,X)$ is defined for all $p \in X$.
In this case
        \begin{displaymath}
        \psi_c(X)=\sup\{\psi_c(p,X):p\in X\}
        \end{displaymath}
is the closed pseudocharacter of the space $X$. Clearly, $\psi_c(p,X) \le \chi(p,X)$ and $\psi_c(X) \le \chi(X)$,
moreover $|X| \le \varrho(X)^{\psi_c(X)}$, where $\varrho(X)$ is the number of regular closed sets in
a $T_2$-space $X$.

\begin{corollary}\label{cr:LXleFpsi}
If $X$ is any $T_2$-space then
      \begin{displaymath}
      \lin(X)\le {2^{\FF(X)\PSIC(X)}} (\le 2^{\FF(X)\chi(X)}).
      \end{displaymath}
      \end{corollary}

\begin{proof}
For any $S \subs X$ we have $\varrho(\overline{S}) \le 2^{|S|}$, and so  by our previous remark
$$\lin(\overline{S}) \le |\overline{S}| \le
\varrho(\overline{S})^{ \PSIC(\overline{S})} \le 2^{|S|\cdot \PSIC(X)}.$$ It immediately follows then that $\mu(X) \le 2^{\FF(X)\PSIC(X)}$, hence
$$\lin(X)\le \FF(X)\cdot \mu(X) \le {2^{\FF(X)\PSIC(X)}}$$ by Theorem \ref{tm:LsupLsmall}.
\end{proof}

As a further corollary we obtain the following non-trivial upper bound for the size of a $T_2$-space space involving its  free set number.

\begin{corollary}\label{cr:XleFpsi}
If $X$ is any $T_2$-space then
\begin{displaymath}
|X|\le 2^{2^{\FF(X)\PSIC(X)}}.
\end{displaymath}
\end{corollary}

\begin{proof}
Let us put  ${\kappa}={\FF(X)\PSIC(X)}$ and consider an elementary submodel
$M$ of $H(\vartheta)$ for a large enough regular cardinal $\vartheta$ such that $|M| = 2^{2^{\kappa}}$, $\,M$ is
$2^{\kappa}$-closed, and $X \in M$. We shall show that $X \subs M$.

Assume, on the contrary, that there is $p\in X\setm M$.
We may then apply Lemma \ref{lm:freeinB} for $A=M\cap X$ and
$\mc W=\{W\in \tau(X)\cap M: p\notin W\}$, but with $\kappa^+$ instead of $\kappa$.
Indeed, condition (a) of Lemma \ref{lm:freeinB}
holds because $M$ and hence $\mc W$ is even $2^{\kappa}$-closed.
Condition (b) holds by elementarity.

To see that (c) holds, first note that
if $S\in {[A]}^{{\kappa}}$ then $S \in M$ as $M$ is ${\kappa}$-closed,
and hence $\overline S \in M$ as well. But then  $|\overline{S}|\le 2^{2^\kappa}$
implies that $\overline{S} \subs A$. For every point $x \in \overline{S}$ we have
$\psi(x,X) \le 2^\kappa$, hence there is a family $\mc V_x \subs \tau(X)$ with $|\mc V_x| \le 2^\kappa$ and $\bigcap \mc V_x = \{x\}$.
By elementarity and the $2^{\kappa}$-closure of $M$ we may assume that $\mc V_x \subs M$. For each $x \in \overline{S}$ there is $W_x \in \mc V_x$
with $p \notin W_x$, hence $\mc W$ covers $\overline S$.
Lemma \ref{cr:LXleFpsi} and the $2^{\kappa}$-closure of $\mc W$ then imply that
there is $W\in \mc W$ with $\overline S\subs W$, thus (c) is indeed satisfied.

But applying Lemma \ref{lm:freeinB} we would obtain a free set in $X$ of size ${\kappa}^+ > \FF(X)$, a contradiction.
Consequently, we indeed have $X \subs M$ and so $|X| \le 2^{2^\kappa} = 2^{2^{\FF(X)\PSIC(X)}}$.
\end{proof}

An alternative way of proving Corollary \ref{cr:XleFpsi} is to use the following inequality due to the first author:
$$|X| \le 2^{\FF(X)\psi(X)\lin(X)}$$ for any $T_2$-space $X$, see \cite[Theorem 3.1.]{SS}.  Indeed, this together with
Corollary \ref{cr:LXleFpsi} immediately yields Corollary \ref{cr:XleFpsi}. We chose to give our above direct proof to
make our presentation self-contained.

\smallskip

Before presenting our next result we need to introduce some new notation.

\begin{definition}
Given a topological space $X$  and a subset $A\subs X$ we let
\begin{enumerate}[(a)]
\item $\mc F(A,X)=\{S\subs A:\text{$S$ is free in $X$}\}$,
\item ${\mu}(A,X)=\sup\{\lin(\overline {S}):S\in \mc F(A,X)\}$,
\item ${\nu}(A,X)=\sup\{|\overline{S}|:S\in \mc F(A,X)\}$.
\end{enumerate}
\end{definition}

\begin{theorem}\label{tm:largepsi}
Assume that $X$ is a $T_1$ space and  $A\subs X$ is its subspace.
If $\lambda$ is a cardinal satisfying the following two conditions:
\begin{enumerate}[(i)]
\item ${\nu}(A,X)\le {\lambda}={\lambda}^{\FF(X)\cdot {\mu(X)}}$,
\item $\forall x\in X$  $\psi(x,A\cup\{x\})\le {\lambda}$,
\end{enumerate}
then $|A|\le {\lambda}$.
\end{theorem}

\begin{proof}
Let us put $\kappa = \FF(X)\cdot {\mu(X)}$ and
let $\mc M$ be a $\lambda$-sized and $\kappa$-closed elementary submodel of
$H(\vartheta)$ for a large enough regular cardinal
$\vartheta$  with $X,A, \lambda \in \mc M$  and  ${\lambda}\subs \mc M$.
We shall show that $A \subs \mc M$. Assume otherwise, then working towards a contradiction
we may fix a point $y\in A\setm \mc M$.
Let us then put $$\mc W=\{W\in \tau(X)\cap \mc M : y \notin W \}.$$

By condition (ii), for each point $x\in X\cap \mc M$ we may fix a family
$\mc U_x \subs \tau(X)$ with $\mc U_x \in \mc M$ and $|\mc U_x| \le {\lambda}$
such that $x\in \bigcap \mc U_x$ and $A\cap \bigcap \mc U_x \subs \{x\}$.
Since ${\lambda}\subs \mc M$, we then also have $\mc U_x\subs \mc M$.

Now, it is straightforward to check that we may apply Lemma \ref{lm:freeinB}
with the parameters $X,\,A \cap \mc M,\, \mc W,$ and $\kappa^+$.

Indeed, conditions (a) and (b) of Lemma \ref{lm:freeinB} hold trivially.
To check (c), consider an arbitrary set $S\subs A \cap \mc M$
that is free in $X$.
Then $S \in \mc M$ and so $\overline S\in \mc M$ as well, moreover
by $|\overline S|\le {\lambda}$ we also have $\overline S \subs \mc M$.
For each $a\in \overline S$ we can pick $U_a\in \mc U_a$ such that $y\notin U_a$.
Then $\mc V=\{U_a:a\in \overline S\}\subs \mc W$ covers $\overline S$.
But $\lin(\overline S)\le {\mu(X)}$, hence
there is $\mc V'\subs \mc V$ with  $|\mc V'|\le \mu(X)$ which also covers $\overline S$.
Then  $\mc V'\in \mc M$ because $\mc M$ is
${\mu}(X)$-closed. Thus $W=\bigcup \mc V'\in \mc W$ with $\overline{S}\subs W$,
showing that condition (c) holds as well.

But then by Lemma \ref{lm:freeinB} there is
a free set in $X$ of cardinality ${\kappa}^+ > \FF(X)$, a contradiction.
Thus, indeed, we have $A \subs \mc M$.
\end{proof}

Let us note that in the case when $A = X$ condition (ii) of Theorem \ref{tm:largepsi}
simply says that $\psi(X) \le \lambda$, and the conclusion is that $|X| \le \lambda$,
hence these two inequalities are equivalent. But if $X$ is $T_3$ then
on one hand from part (ii) of Corollary \ref{cr:Lle22F} we have $\mu(X) \le 2^{\FF(X)}$,
while on the other hand we trivially have $\nu(X) \le 2^{2^{\FF(X)}}$.
Thus, in this case we may apply our theorem with the choice
$\lambda = 2^{2^{\FF(X)}}$ to conclude the following somewhat surprising result.

\begin{corollary}\label{cr:T3psi}
For any $T_3$-space $X$ we have $$ \psi(X) \le 2^{2^{\FF(X)}}\,\, \Leftrightarrow \,\, |X| \le 2^{2^{\FF(X)}}.$$
\end{corollary}

The real significance of Theorem \ref{tm:largepsi}, however, will become apparent in the next section.

\section{The free set number of the $G_\delta$-modification}

The $G_\delta$-modification $X_\delta$ of a topological space $X$ is the space on the same underlying set
generated by, i.e. having as a basis, the collection of all $G_\delta$ subsets of $X$.
The aim of this section is to give estimates of $\FF(X_\delta)$ in terms of $\FF(X)$.

Let us recall before giving our next result that $p$ is a {\em complete accumulation point (CAP)}
of a subset $A$ of a space $X$ if for every neighborhood $U$ of $p$ in $X$ we have $|U \cap A| = |A|$.
We shall denote by $A^\circ$ the set of all CAPs of $A$.

\begin{lemma}\label{lm:circ}
\begin{enumerate}[(i)]
\item Let $X$ be any space and $\varrho > \mu(X)$ be a regular cardinal.
Then $A^\circ \ne \emptyset$ for every set $A \in [X]^\varrho$.

\smallskip

\item Let $X$ be a $T_2$-space, $\varrho > 2^{2^{\FF(X)}}$ be a regular cardinal, and
$s = \langle x_\alpha : \alpha < \varrho \rangle$
be a one-one $\varrho$-sequence in $X$ such that
$$\{x_\alpha : \alpha < \varrho\}^\circ = \{x\}$$ for some $x \in X$.
Then $s$ is {\em not a free sequence} in $X_\delta$.
\end{enumerate}
\end{lemma}

\begin{proof}
(i) If $A^\circ$ would be empty then we could apply Theorem \ref{lm:freeinB} with parameters
$X,\, A,\, \mc W,\,$ and $\FF(X)^+$, where $\mc W = \{W \in \tau(X) : |W \cap A| < \varrho\}$,
to obtain a subset of $A$ of size $\FF(X)^+$ that is free in $X$, a contradiction.

\smallskip

(ii) Since $X$ is $T_2$ we now have $\mu(X) \le 2^{2^{\FF(X)}} < \varrho$, hence by part (i)
the assumption $\{x_\alpha : \alpha < \varrho\}^\circ = \{x\}$ implies that
the sequence $s$ actually converges to the point $x$. Clearly, then $s$ also converges to $x$ in $X_\delta$.

Let us now put $A = \{x_\alpha : \alpha < \varrho\}$ and $\mc W = \{W \in \tau(X) : |W \cap A| < \varrho\}$,
as in part (i). We may then apply Theorem \ref{lm:freeinB} again, this time with the parameters
$X \setm \{x\},\, A,\, \mc W,\,$ and $\FF(X)^+$, to obtain a subset  $B$ of $A$ of size $\FF(X)^+$ that is free
not in $X$ but in $X \setm \{x\}$. However, then for every open neighbourhood $U$ of $x$ we must have
$|B \setm U| \le \FF(X)$ because $B \setm U$ is already a free set in $X$. This clearly also implies that
$|B \setm H| \le \FF(X)$, hence $B \cap H \ne \emptyset$ for any $G_\delta$-set $H$ containing $x$,
consequently $x \in \overline{B}^\delta$, the $G_\delta$-closure of $B$. But $|B| = \FF(X)^+ < \varrho$ then
implies that $B$ is included in a proper initial segment of $s$, while $x$ is in the
the $G_\delta$-closure of all final segments of $s$, consequently $x$ witnesses that
$s$ is not a free sequence in $X_\delta$.
\end{proof}

Our next result is about half-way towards our goal, a general estimate of the
free set number $\FF(X_\delta)$ in terms of $\FF(X)$ for any $T_2$-space $X$.

\begin{theorem}\label{tm:plus}
Let $X$ be a $T_2$-space and $\lambda$ be a cardinal such that $|X| = \lambda^+$ and
$${\lambda}={\lambda}^{\FF(X)\cdot {\mu(X)}} \ge  2^{2^{\FF(X)}}.$$
Then $\FF(X_\delta) \le \lambda$.
\end{theorem}

\begin{proof}
Let us consider an arbitrary one-one sequence $s = \langle x_\alpha : \alpha < \lambda^+ \rangle$
of length $\lambda^+$ in $X$, we shall show that $s$ is not free in $X_\delta$. Let us
put $A = \{x_\alpha : \alpha < \lambda^+\}$, we may then apply
Theorem \ref{tm:largepsi} to obtain a point $x \in X$ such that $\psi(x,A\cup\{x\}) = {\lambda}^+$.
Since $\psi_c(x,X) \le |X| = \lambda^+ $ holds trivially, we then have $\psi_c(x,X) = \lambda^+$
as well, so we may fix open neighbourhoods $\{W_\xi : \xi < \lambda^+\}$ of $x$ with
$$\bigcap \{\overline{W_\xi} : \xi < \lambda^+\} = \{x\}.$$

This allows us to obtain a set $B \in [A]^{\lambda^+}$ such that $B^\circ = \{x\}$.
Indeed, by transfinite recursion on $\xi < \lambda^+$ we shall pick
distinct points $b_\xi \in A \setm \{x\}$ and open neighbourhoods $U_\xi$ of $b_\xi$ and $V_\xi \subs W_\xi$ of $x$
with $U_\xi \cap V_\xi = \emptyset$ as follows.
If $\{b_\eta : \eta < \xi\}$ have been chosen then we use $\psi(x,A\cup\{x\}) = {\lambda}^+$ to pick the point
$b_\xi \in A \cap \bigcap \{V_\eta : \eta < \xi\} \setm \{x\}$ and then, using that $X$ is $T_2,$ we can easily choose $U_\xi$ and $V_\xi$
appropriately. Let us put $B = \{b_\xi : \xi < \lambda^+\}$, then
it's obvious from the construction that $B^\circ \subs \{x\}$. But then by $\lambda^+ > 2^{2^{\FF(X)}}$
and part (i) of Lemma \ref{lm:circ} $B^\circ \ne \emptyset$, hence we actually have $B^\circ = \{x\}$.
But then by part (ii) of Lemma \ref{lm:circ} the cofinal subsequence of $s$ with range $B$ is not free in
$X_\delta$, and so neither is $s$.
\end{proof}

We are now ready to present the main result of this section.

\begin{theorem}\label{tm:ultimate-thm}
For every $T_2$-space $X$ we have
\begin{displaymath}
\FF(X_{\delta})\le \lambda(X) = 2^{\mu(X) \cdot 2^{\FF(X)}} \cdot {\nu}(X)^{{\mu}(X) \cdot \FF(X)^+}.
\end{displaymath}
\end{theorem}

\begin{proof}
Our proof is indirect, so assume that $\FF(X_{\delta}) > \lambda(X)$, hence there is a sequence
$s = \langle x_\alpha : \alpha < \lambda(X)^+ \rangle$ of length $\lambda(X)^+$ that is free in $X_\delta$.
This means that there is a function $$h : X \times \lambda(X)^+ \to [\tau(X)]^\omega$$ such that
for each $x \in X$ and $\alpha < \lambda(X)^+$ we have $x \in \bigcap h(x,\alpha)$ and
$$\bigcap h(x,\alpha) \cap \{x_\beta : \beta < \alpha\} = \emptyset\,\,\text{ or }\,\,\bigcap h(x,\alpha) \cap \{x_\beta : \beta \ge \alpha\} = \emptyset.$$

To simplify our notation, let us put for the rest of our proof $\kappa = {\mu}(X) \cdot \FF(X)^+$. Then clearly $\lambda(X)^\kappa = \lambda(X)$,
hence also $(\lambda(X)^+)^\kappa = \lambda(X)^+$ hold. Consequently we may fix a $\kappa$-closed elementary submodel $M$ of
$H(\vartheta)$ for a large enough regular cardinal $\vartheta$  with $|M| = \lambda(X)^+$ such that $X,\,s,\,h,\,\lambda(X)^+ \in M$ and $\lambda(X)^+ \subs M$.

We then consider the $T_2$-space $X_M = \langle X \cap M,\,\tau_M\rangle$, where $\tau_M$ is the topology on $X \cap M$ generated by $\{U \cap M : U \in \tau(X) \cap M\}$, see e.g.
\cite{JT}. Then from $h \in M$ and $h \subs M$ it clearly follows that $s$ is a free sequence in $(X_M)_\delta$ as well, hence $$\FF\big((X_M)_\delta \big) = |X_M| = \lambda(X)^+.$$

Our next aim is to show that $\FF(X_M) = \FF(X)$. To see this, consider first any sequence $z = \langle y_\xi : \xi < \FF(X)^+ \rangle$ in $X \cap M$.
Then the $\kappa$-closure of $M$ implies $z \in M$, moreover $z$ is not free in $X$, hence by elementarity there is a point $x \in X \cap M$
that witnesses this, i.e. there is $\xi < \FF(X)^+$ such that for any $x \in U \in \tau(X)$ we have
$$\{y_\eta : \eta < \xi \} \cap U \ne \emptyset \ne \{y_\eta : \eta \ge \xi \} \cap U.$$ This clearly implies that
$x$ is a witness for  $z$ not being free in $X_M$ either,
hence we have $\FF(X_M) \le \FF(X)$.

On the other hand, if $\sigma$ is a cardinal for which there is a free sequence of length $\sigma$ in $X$ then, by elementarity again,
there is also such a free sequence $z = \langle y_\xi : \xi < \sigma \rangle$ in $X$
of length $\sigma$ that is a member and hence a subset of $M$.
Note however, that for every set $S \subs X$ if $S \in M$ and $|S| \le \FF(X)$ then we have both $\overline{S} \in M$
and $|\overline{S}| \le 2^{2^{\FF(X)}} \le \lambda(X)$, hence $\overline{S} \subs M$. These together imply $\overline{S}^M = \overline{S}$,
and so that $z$ remains a free sequence in $X_M$. Thus we also have $\FF(X_M) \ge \FF(X)$, consequently $\FF(X_M) = \FF(X)$.

Completely analogous arguments yield us that we also have $\mu(X_M) = \mu(X)$ and $\nu(X_M) = \nu(X)$.
But then we see that $X_M$ and $\lambda(X) = \lambda(X_M)$ satisfy the assumptions of Theorem \ref{tm:plus},
consequently $\FF\big((X_M)_\delta \big) \le \lambda(X)$, and so we arrived at a contradiction.

\end{proof}

The formulation of Theorem \ref{tm:ultimate-thm}, namely the choice of $\lambda(X)$, might not be simple or esthetic
but the following corollary is. It follows immediately from the simple observations that
$\mu(X) \le \nu(X) \le 2^{2^{\FF(X)}}$ if $X$ is $T_2$ and that $\mu(X) \le 2^{\FF(X)}$ if $X$ is even $T_3$.

\begin{corollary}\label{cr:ffdelta}
For every $T_2$-space $X$ we have
\begin{displaymath}
\FF(X_{\delta})\le 2^{2^{2^{\FF(X)}}}.
\end{displaymath}
Moreover, if $\mu(X) \le 2^{\FF(X)}$, in particular if $X$ is $T_3$, then
\begin{displaymath}
\FF(X_{\delta}) \le 2^{2^{\FF(X)}}.
\end{displaymath}
\end{corollary}

\end{document}